\documentclass{article}
\usepackage[affil-it]{authblk}
\usepackage[utf8]{inputenc}
\usepackage{amsmath, amsfonts, amsthm, amssymb, latexsym}
\usepackage{geometry}
\usepackage{enumerate}
\usepackage{multirow}
\usepackage{rotating}
\usepackage{graphicx,color}
\usepackage{float}
\usepackage{mathrsfs}

\newtheorem{thm}{Theorem}[section]
\newtheorem{lemma}[thm]{Lemma}

\newtheorem{prop}[thm]{Proposition}

\title{On inverse problems for uncoupled space-time fractional operators
	involving time-dependent coefficients}
\author{Li Li}
\affil{Institute for Pure and Applied Mathematics, University of California,\\
	Los Angeles, CA 90095, USA}
\date{}

\begin{document}
	\maketitle
	
	\noindent \textbf{ABSTRACT.}\, We study the uncoupled space-time fractional operators involving time-dependent coefficients and formulate the corresponding inverse problems. Our goal is to determine the variable coefficients from the exterior partial measurements of the Dirichlet-to-Neumann map. We exploit the integration by parts formula for Riemann-Liouville and Caputo derivatives to derive the Runge approximation property for our space-time fractional operator based on the unique continuation property of the fractional Laplacian. This enables us to extend early unique determination results for space-fractional but time-local operators to the space-time fractional case.

	\section{Introduction}
	The basic uncoupled space-time fractional diffusion equation
	\begin{equation}\label{basicSTEq}
		\partial^\alpha_t u+ (-\Delta)^s u= 0,\qquad 0< \alpha< 1,\,\,
		0< s< 1
	\end{equation}
	models anomalous diffusion. Here the Caputo derivative $\partial^\alpha_t$
	(time-fractional derivative) describes particle trapping phenomena; The fractional Laplacian $(-\Delta)^s$ (space-fractional derivative)
	describes long particle jumps. See \cite{meerschaert2002governing} for more background information on (\ref{basicSTEq}). See \cite{chen2012space} for a probabilistic interpretation for (\ref{basicSTEq}).
	
	In this paper, we consider uncoupled space-time fractional operators involving time-dependent coefficients which generalize $\partial^\alpha_t+ (-\Delta)^s$
	and formulate the corresponding inverse problems.
	
	\subsection{Main results}
	We will first study the initial exterior problem
	\begin{equation}\label{Cafracspacetime}
		\left\{
		\begin{aligned}
			\partial^\alpha_t u+ \mathcal{R}^s_{A(t)} u + q(t)u&= 0
			\quad \,\,\,\Omega\times (0, T)\\
			u&= g\quad \,\,\,\Omega_e\times (0, T)\\
			u&= 0\quad \,\,\,\Omega\times \{0\}
		\end{aligned}
		\right.
	\end{equation}
	where $\Omega$ is a bounded domain and $\Omega_e:= \mathbb{R}^n\setminus \Omega$. 
	Here the time-dependent fractional operator 
	$\mathcal{R}^s_{A(t)}$ is formally defined by
	\begin{equation}\label{RsAt}
		\mathcal{R}^s_{A(t)} u(x):= \lim_{\epsilon\to 0^+}
		\int_{\mathbb{R}^n\setminus B_\epsilon(x)}(u(x)- R_{A(t)}(x, y)u(y))K(x,y)\,dy
	\end{equation}
	where
	\begin{equation}\label{Kxy}
		K(x, y):= c_{n, s}/|x-y|^{n+2s},
	\end{equation}
	$A(\cdot, t)$ is a time-dependent real vector-valued magnetic potential and
	\begin{equation}\label{cost}
		R_{A(t)}(x, y):= \cos((x-y)\cdot A(\frac{x+y}{2}, t)).
	\end{equation}
	The operator $\mathcal{R}^s_{A(t)}$ coincides with the fractional Laplacian $(-\Delta)^s$ when $A\equiv 0$.
	
	Under appropriate regularity and support assumptions on magnetic potential $A$, electric potential $q$ and the exterior data $g$, 
	(\ref{Cafracspacetime}) is well-posed, so we can define the associated  Dirichlet-to-Neumann map by
	\begin{equation}\label{tDN}
		\Lambda_{A, q}g:= \mathcal{R}^s_A u_g|_{\Omega_e\times (0, T)}.
	\end{equation}
	
	Our goal here is to determine both $A$ and $q$ from the exterior partial measurements of $\Lambda_{A, q}$.
	The following theorem is the first main result in this paper, which can be viewed as a space-time fractional analogue of Theorem 1.1 in \cite{li2021fractional}.
	
	\begin{thm}
		Suppose $\Omega\subset B_r(0)$ for some constant $r> 0$, $\mathrm{supp}\,A_j(t)\subset\Omega$ for $t\in [0, T]$, $A_j\in C([0, T]; L^\infty(\mathbb{R}^n))$,
		$q_j\in C([0, T]; L^\infty(\Omega))$, 
		$W_j$ are open sets s.t. $W_j\cap B_{3r}(0)= \emptyset$
		($j= 1, 2$). Let
		$$W^{(1, 2)}= \{\frac{x+ y}{2}: x\in W_1, y\in W_2\}.$$
		Also assume $W^{(1, 2)}\setminus \Omega\neq \emptyset$. If 
		\begin{equation}\label{IDN}
			\Lambda_{A_1, q_1}g|_{W_2\times (0, T)}
			= \Lambda_{A_2, q_2}g|_{W_2\times (0, T)}
		\end{equation}
		for any $g\in C^\infty_c(W_1\times (0, T))$, then $A_1(t)= \pm A_2(t)$ and $q_1= q_2$ in $\Omega\times (0, T)$.
	\end{thm}
	We remark that the seemingly unnatural assumptions on $W_j$ in the statement are necessary (see the remark after Theorem 1.1 in \cite{li2021fractional}). However, if we replace $\mathcal{R}^s_{A(t)}$ by $(-\Delta)^s$ in (\ref{Cafracspacetime}), only interested in the determination of $q$, then the assumptions on $W_j$ in the statement can be simply replaced by $W_j\subset \Omega_e$ (see Subsection 3.2 for details).
	
	We will next study the semilinear problem
	\begin{equation}\label{Semifracspacetime}
		\left\{
		\begin{aligned}
			\partial^\alpha_t u+ (-\Delta)^s u + a(x, t, u)&= 0
			\quad \,\,\,\Omega\times (0, T)\\
			u&= g\quad \,\,\,\Omega_e\times (0, T)\\
			u&= 0\quad \,\,\,\Omega\times \{0\}
		\end{aligned}
		\right.
	\end{equation}
	where the nonlinearity satisfies
	\begin{equation}\label{fracpower}
		a(x, t, z)= \sum^m_{k=1}a_k(x, t)|z|^{b_k}z,
	\end{equation}
	$a_k\geq 0$ are smooth in $\bar{\Omega}\times [0, T]$, $b_1= 0$ and the powers $0< b_2<\cdots< b_m$ are not necessarily integers.
	We will see that the Dirichlet-to-Neumann map 
	\begin{equation}\label{semiDN}
		\Lambda_a: g\to (-\Delta)^su_g|_{\Omega_e\times (0, T)}
	\end{equation}
	is well-defined at least for $g\in C^\infty_c(\Omega_e\times (0, T))$.
	
	Our goal here is to determine the nonlinearity $a$ from the exterior partial measurements of $\Lambda_a$.
	The following theorem is the second main result in this paper, which can be viewed as a space-time fractional analogue of Theorem 1.1 in \cite{li2021inversediff}.
	
	\begin{thm}
		Let $W_1, W_2\subset \Omega_e$ be nonempty and open. If
		\begin{equation}\label{semiIDN}
			\Lambda_{a^{(1)}} g|_{W_2\times (0, T)}= \Lambda_{a^{(2)}} g|_{W_2\times (0, T)},
			\qquad g\in C^\infty_c(W_1\times (0, T)),
		\end{equation}
		then $a^{(1)}_k= a^{(2)}_k$ in $\Omega\times (0, T)$, $k= 1,2,\cdots,m$.
	\end{thm}
	
	\subsection{Connection with earlier literature}
	So far there have been many contributions in the study of inverse problems for time-fractional but space-local operators. See, for instance, \cite{kian2018global} for a work in this direction where the authors determined various time-independent smooth
	coefficients appearing in the equation from the knowledge of the associated  Dirichlet-to-Neumann map based on the inverse spectral theory.
	
	The study of (Calder\'on type) inverse problems for space-fractional operators dates back to \cite{ghosh2020calderon}. See \cite{ghosh2017calderon, ghosh2020uniqueness, ghosh2021calder, covi2022higher, li2021determining, li2022elastic} for further results in this direction.
	The proof of unique determination results in all these works heavily relies on exploiting the unique continuation property of the fractional Laplacian, which makes inverse problems for space-fractional operators often more manageable than their space-local counterparts. We adopt this framework here as well.
	Our approaches will be mainly based on the ones introduced in \cite{li2021fractional, li2021inversediff}. 
	
	To the best knowledge of the author, there are very few existing rigorous works on inverse problems for space-time fractional operators. Calder\'on type problems for the coupled space-time fractional operator $(\partial_t-\Delta)^s$
	has been studied in \cite{lai2020calderon, banerjee2022calder}. This fractional evolutionary operator actually behaves more like a fractional elliptic operator, which distinguishes it from our 
	uncoupled space-time fractional operators. Another work in this direction can be found in \cite{helin2020inverse} where the authors determined the Riemannian metric appearing in the uncoupled space-time fractional equation up to an isometry from the knowledge of the associated source-to-solution map in the setting of closed manifolds (also see \cite{feizmohammadi2021fractional, quan2022calder} for such geometric inverse problems for other fractional operators).
	
	We also mention that our first inverse problem involving a magnetic potential can be viewed as a space-time fractional analogue of the classical magnetic Calder\'on problem studied in \cite{nakamura1995global, ferreira2007determining, krupchyk2014uniqueness}. See \cite{covi2020inverse, lai2021inverse} for the study of Calder\'on type problems for a different kind of fractional magnetic operators.
	
	\subsection{Organization}
	The rest of this paper is organized in the following way. In Section 2, we will summarize the background knowledge. In Section 3, we will first show the well-posedness of (\ref{Cafracspacetime}); Then we will derive the integral identity for the Dirichlet-to-Neumann maps and the Runge approximation property
	of the space-time fractional operator to prove Theorem 1.1. In Section 4, we will first prove an a priori $L^\infty$ estimate and the well-posedness of (\ref{Semifracspacetime}); Then we will derive a linearization result to prove Theorem 1.2.
	\medskip
	
	\noindent \textbf{Acknowledgements.} The author would like to thank Professor Gunther Uhlmann for helpful discussions.
	
	\section{Preliminaries}
	Throughout this paper we use the following notations.
	
	\begin{itemize}
		\item Fix the space dimension $n\geq 2$.
		
		\item Fix the fractional powers $0< \alpha< 1$ and $0< s< 1$.
		
		\item $\Omega$ denotes a bounded domain with smooth boundary and
		$\Omega_e:= \mathbb{R}^n\setminus\bar{\Omega}$.
		
		\item $B_r(0)$ denotes the open ball centered at the origin with radius $r> 0$
		in $\mathbb{R}^n$.
		
		\item $c, C, C', C_1,\cdots$ denote positive constants.
		
		\item $X$ denotes a Banach space and $X^*$ denotes the continuous dual space of $X$.
		
		\item $\langle f, u\rangle$ denotes the standard $L^2$ distributional pairing
		between $f$ and $u$ if $u$ (resp., $f$) is a (spatial) $n$-variable function (resp., functional).
		
		\item For functions $f(t), g(t)$ ($t> 0$), we use the standard convolution notation
		$$(f* g)(t)= \int^t_0 f(\tau)g(t-\tau)\,d\tau.$$
	\end{itemize}
	
	\subsection{Function spaces}
	Throughout this paper we refer all function spaces to real-valued function spaces. 
	We use $H^r$ to denote the standard Sobolev space $W^{r,2}$.
	We use $C^\alpha$ to denote the standard H\"older space $C^{0, \alpha}$.
	
	Let $U$ be an open set in $\mathbb{R}^n$. Let $F$ be a closed set in $\mathbb{R}^n$. Then
	$$H^r(U):= \{u|_U: u\in H^r(\mathbb{R}^n)\},\qquad 
	H^r_F(\mathbb{R}^n):= 
	\{u\in H^r(\mathbb{R}^n): \mathrm{supp}\,u\subset F\},$$
	$$\tilde{H}^r(U):= 
	\mathrm{the\,\,closure\,\,of}\,\, C^\infty_c(U)\,\,\mathrm{in}\,\, H^r(\mathbb{R}^n).$$
	
	For $r\in\mathbb{R}$, we have the natural identifications
	$$H^{-r}(\mathbb{R}^n)= H^r(\mathbb{R}^n)^*,\quad \tilde{H}^r(\Omega)= H^r_{\bar{\Omega}}(\mathbb{R}^n),\quad 
	H^{-r}(\Omega)= \tilde{H}^r(\Omega)^*.$$
	
	We use $C([0, T]; X)$ (resp., $C^\alpha([0, T]; X)$)
	to denote the space consisting of the corresponding Banach space-valued continuous (resp., $C^\alpha$-continuous) functions on $[0, T]$. 
	$L^2(0, T; X)$ (resp., $H^1(0, T; X)$) denotes the space consisting of the corresponding Banach space-valued $L^2$-functions (resp., $H^1$-functions).

	\subsection{Riemann-Liouville and Caputo derivatives}
	Throughout this paper we use the notation
	$$\phi_\alpha(t):= \frac{t^{\alpha-1}}{\Gamma(\alpha)},\qquad t>0$$
	where $\Gamma$ is the standard Gamma function. It is straightforward to verify that
	\begin{equation}\label{alpha1}
		\phi_\alpha * \phi_{1-\alpha}= 1,\qquad t>0.
	\end{equation}
	
	The left Riemann-Liouville fractional integral of order $\alpha$ is defined by
	$$I^\alpha_{0,t}u:= \phi_\alpha* u$$
	and the left Riemann-Liouville fractional derivative of order $\alpha$ is defined by
	$$D^\alpha_{0,t}u:= \partial_t(I^{1-\alpha}_{0,t}u)
	= \partial_t(\phi_{1-\alpha}* u).$$
	Correspondingly, the right Riemann-Liouville fractional integral of order $\alpha$ is defined by
	$$I^\alpha_{t,T}u:= \frac{1}{\Gamma(\alpha)}\int^T_t (\tau- t)^{\alpha-1}u(\tau)\,d\tau$$
	and the right Riemann-Liouville fractional derivative of order $\alpha$ is defined by
	$$D^\alpha_{t,T}u:= -\partial_t(I^{1-\alpha}_{t,T}u).$$
	
	The (left) Caputo fractional derivative of order $\alpha$ is defined by
	$$\partial^\alpha_t u(t):= D^\alpha_{0,t}(u(t)- u(0)).$$
	It is straightforward to verify that $\partial^\alpha_t u= I^{1-\alpha}_{0,t}(\partial_t u)$ if $u$ is sufficiently regular.
	
	For (scalar-valued) functions $f(t), g(t)$, we have the following integration by parts formula (see, for instance, (3.6) in \cite{Warma2019approx}) 
	\begin{equation}\label{fracibp}
		\int^T_0 g\partial^\alpha_t f= \int^T_0 f D^\alpha_{t,T}g + 
		(f I^{1-\alpha}_{t,T}g)|^{t= T}_{t= 0}.
	\end{equation}
	Here the value of $I^{1-\alpha}_{t,T}g$ at $t=T$ should be interpreted in the limit sense. We also have the following inequality
	(see, for instance, Lemma 2.1 in \cite{vergara2017stability})
	\begin{equation}\label{Hconvex}
		\partial^\alpha_t(H(f(t))\leq H'(f(t))\partial^\alpha_t f(t)
	\end{equation}
	where $H\in C^1(\mathbb{R})$ is convex. We note that both (\ref{fracibp})
	and (\ref{Hconvex}) can be generalized for Banach space-valued $f(t), g(t)$
	if we replace the pointwise product by the distributional pairing.
	
	Let $L$ be a sectorial operator (i.e. a closed, densely defined linear operator which generates an analytic semigroup) in $X$. For sufficiently regular $f$, we can take Laplace transform to show that 
	\begin{equation}\label{SPformula}
		u(t)= S_{\alpha}(t)u_0+ \int^t_0 P_{\alpha}(t-\tau)f(\tau)\,d\tau
	\end{equation}
	satisfies the continuity at $t=0$, and this formula gives the solution of
	\begin{equation}\label{CdeL}
		\partial^\alpha_t u= Lu+ f,\qquad u(0)= u_0;
	\end{equation}
	We can also take Laplace transform to show that 
	\begin{equation}\label{PPformula}
		u(t)= P_{\alpha}(t)h_0+ \int^t_0 P_{\alpha}(t-\tau)f(\tau)\,d\tau
	\end{equation}
	satisfies that $I^\alpha_{0,t}u$ is continuous at $t=0$, and this formula gives the solution of
	\begin{equation}\label{RLdeL}
		D^\alpha_{0,t} u= Lu+ f,\qquad I^\alpha_{0,t}u(0)= h_0.
	\end{equation}
	Here $\{S_{\alpha}(t)\}$, $\{P_{\alpha}(t)\}$ are the resolvent families associated with $L$, which can be expressed by integrals involving the semigroup $\{S(t)\}$
	generated by $L$ and the Wright type function $\Phi_\alpha(t)$. Since we are not going to use the explicit expressions of $S_{\alpha}(t), P_{\alpha}(t)$, we refer interested readers to (2.1.9) in \cite{gal2020fractional} for details.
	
	\subsection{Space-fractional operators}
	It is well-known that the fractional Laplacian $(-\Delta)^s$ can be defined by
	
	\begin{equation}\label{fracLa}
		(-\Delta)^s u(x):= c_{n, s}\lim_{\epsilon\to 0^+}
		\int_{\mathbb{R}^n\setminus B_\epsilon(x)}\frac{u(x)-u(y)}{|x-y|^{n+2s}}\,dy
	\end{equation}
	and for each $r\in \mathbb{R}$, we have
	$$(-\Delta)^s: H^r(\mathbb{R}^n)\to H^{r-2s}(\mathbb{R}^n).$$
	
	We will see that many properties of the fractional Laplacian are preserved under the perturbation by the magnetic potential $A$.
	Throughout this paper we assume 
	$$A\in C([0, T]; L^\infty(\mathbb{R}^n)),\qquad\mathrm{supp}\,A(t)\subset \Omega$$
	for each $t\in [0,T]$ and 
	$q\in C([0, T]; L^\infty(\Omega))$. 
	
	It has been shown that (see (1) and (6) in \cite{li2021fractional}) for $u, v\in H^s(\mathbb{R}^n)$,
	\begin{equation}\label{biRsAt}
		\langle \mathcal{R}^s_{A(t)}u, v\rangle:= 
		2\int_{\mathbb{R}^n}\int_{\mathbb{R}^n}(u(x)- R_{A(t)}(x, y)u(y))v(x)K(x,y)\,dxdy,
	\end{equation}
	and $\mathcal{R}^s_{A(t)}$ is symmetric:
	\begin{equation}\label{RsAsym}
		\langle \mathcal{R}^s_{A(t)}u, v\rangle
		= \langle \mathcal{R}^s_{A(t)}v, u\rangle.
	\end{equation}
	
	We define the bilinear form $B_t$
	associated with $A, q$ by
	\begin{equation}\label{tdbf}
		B_t[u, v]:= \langle \mathcal{R}^s_{A(t)}u, v\rangle+ \int_\Omega q(t)uv,\qquad
		t\in [0, T].
	\end{equation}
	
	\begin{lemma}
		$B_t$ is bounded:
		\begin{equation}\label{Bbdd}
			|B_t[u,v]|\leq C_0||u||_{H^s}||v||_{H^s},\qquad u, v\in H^s(\mathbb{R}^n).
		\end{equation}
		$B_t$ is coercive:
		\begin{equation}\label{Bcoer}
			B_t[u, u]\geq c_1||u||^2_{H^s}- c_2||u||^2_{L^2},\qquad 
			u\in \tilde{H}^s(\Omega).
		\end{equation}
		Here $C_0, c_1, c_2$ do not depend on $t$.
	\end{lemma}
	
	The estimates above have been proved in Section 2 in \cite{li2021fractional}.
	
	\begin{prop}
		Suppose $\Omega\subset B_r(0)$ for some $r> 0$, $W$ is a nonempty open set s.t. $$W\cap B_{3r}(0)= \emptyset.$$
		(i)\, For $g\in C^\infty_c(W\times (0, T))$, we have
		$$\mathcal{R}^s_A g|_{\Omega\times (0, T)}
		= (-\Delta)^s g|_{\Omega\times (0, T)};$$
		(ii) (Proposition 2.4 in \cite{li2021fractional}) If
		$$u\in \tilde{H}^s(\Omega),\qquad \mathcal{R}^s_{A(t)}u|_W= 0,$$
		then $u= 0$ in $\mathbb{R}^n$.
	\end{prop}
	It is straightforward to verify (i) based on the support assumptions on $A$ and $g$. (ii) is based on the following unique continuation property of $(-\Delta)^s$ (see Theorem 1.2 in \cite{ghosh2020calderon}). 
	\begin{prop}
		Let $u\in H^r(\mathbb{R}^n)$ for some $r\in \mathbb{R}$. Let $W\subset \mathbb{R}^n$ be nonempty and open. If $$(-\Delta)^su= u= 0\quad\text{in}\,\,W,$$
		then $u= 0$ in $\mathbb{R}^n$.
	\end{prop}
	
	\section{Linear problem}
	\subsection{Forward problem}
	We will study (\ref{Cafracspacetime}) as well as the related problem
	\begin{equation}\label{RLfracspacetime}
		\left\{
		\begin{aligned}
			D^\alpha_{0,t} u+ \mathcal{R}^s_{A(t)} u + q(t)u&= 0
			\quad \,\,\,\Omega\times (0, T)\\
			u&= g\quad \,\,\,\Omega_e\times (0, T)\\
			I^\alpha_{0,t}u&= 0\quad \,\,\,\Omega\times \{0\}
		\end{aligned}
		\right.
	\end{equation}
	for $g\in C^\infty_c(\Omega_e\times (0, T))$.
	
	By using the substitution $u= w+g$,
	(\ref{Cafracspacetime}) can be converted into the initial value problem
	\begin{equation}\label{Caivp}
		\left\{
		\begin{aligned}
			\partial^\alpha_t w+ \mathcal{R}^s_{A(t)} w + q(t)w&= f
			\quad \,\,\,\Omega\times (0, T)\\
			w&= 0\quad \,\,\,\Omega\times \{0\}
		\end{aligned}
		\right.
	\end{equation}
	and (\ref{RLfracspacetime}) can be converted into the (integral) initial value problem
	\begin{equation}\label{RLivp}
		\left\{
		\begin{aligned}
			D^\alpha_{0,t} w+ \mathcal{R}^s_{A(t)} w + q(t)w&= f
			\quad \,\,\,\Omega\times (0, T)\\
			I^\alpha_{0,t}w&= 0\quad \,\,\,\Omega\times \{0\}
		\end{aligned}
		\right.
	\end{equation}
	
	Based on Lemma 2.1 in Subsection 2.3 and Theorem 3.1 in \cite{zacher2009weak}, we have the following well-posedness result.
	\begin{prop}
		Suppose $f\in L^2(0, T; H^{-s}(\Omega))$. 
		Then (\ref{RLivp}) has a unique solution satisfying
		$$w\in L^2(0, T; \tilde{H}^s(\Omega)),\qquad
		I^\alpha_{0,t}w\in H^1(0, T; H^{-s}(\Omega)).$$
		This actually implies 
		$I^\alpha_{0,t}w\in C([0,T]; L^2(\Omega))$ (see Theorem 2.1 in \cite{zacher2009weak}).
	\end{prop}
	
	We note that
	$$0\leq 1- R_{A(t)}(x,y)= 
	2\sin^2(\frac{1}{2}(x-y)\cdot A(\frac{x+y}{2}, t))\leq 
	C_A\min\{1, |x-y|^2\}$$
	and (based on (\ref{fracLa}) and (\ref{RsAt})) we have
	$$|((-\Delta)^s- \mathcal{R}^s_{A(t)}) u(x)|\leq
	\int (1- R_{A(t)}(x, y))K(x, y)|u(y)|dy.$$
	Since we have
	$$\int (1- R_{A(t)}(x, y))K(x, y)dy\leq C,\qquad x\in \mathbb{R}^n,$$
	by the generalized Young’s inequality (see, for instance, Proposition 0.10 in \cite{folland1995introduction}), we get 
	\begin{equation}\label{L2bdd}
		||((-\Delta)^s- \mathcal{R}^s_{A(t)})u||_{L^2(\mathbb{R}^n)}\leq 
		C||u||_{L^2(\mathbb{R}^n)}.
	\end{equation}

	Now we know that the operator ${R}^s_{A(t)} + q(t)$
	can be written as 
	$$(-\Delta)^s+ F(t,\cdot)$$
	where the linear operator
	$$F(t,\cdot):=
	\mathcal{R}^s_{A(t)}- (-\Delta)^s+ q(t)$$
	is $L^2$ bounded (uniformly in $t$). 
	
	We also note that for sufficiently regular $f$, both (\ref{RLivp}) and (\ref{Caivp}) can be written as the integral equation (see (\ref{SPformula}) and (\ref{PPformula}))
	$$w(t)= \int^t_0 P_{\alpha}(t-\tau)(f(\tau)- F(\tau, w(\tau))\tau\,d\tau.$$
	Here $\{P_\alpha(t)\}$ is the resolvent family associated with
	the sectorial operator
	$$L: w\to -(-\Delta)^s w|_\Omega,\qquad D(L)= \{w\in \tilde{H}^s(\Omega):
	(-\Delta)^s w|_\Omega\in L^2(\Omega)\}.$$
	
	Based on Theorem 1.1 in \cite{kojima2021existence}, we have the following well-posedness result.
	\begin{prop}
		Let $f\in C^\alpha([0, T]; L^2(\Omega))$
		Then (\ref{Caivp}) (or (\ref{RLivp})) has a unique solution
		$$w\in C([0,T]; L^2(\Omega)).$$
	\end{prop}
	We remark that the well-posedness result above actually holds true for the general Lipschitz type nonlinear map $F(t,\cdot)$ and the regularity assumption on $f$ can be weakened. We refer readers to \cite{kojima2021existence} for more details.
	
	In the rest of this section, we will assume $\Omega\subset B_r(0)$ for some constant $r> 0$ and $W$ is a nonempty open set in $\mathbb{R}^n$ s.t. $W\cap B_{3r}(0)= \emptyset$ unless otherwise stated.
	We note that for $g\in C^\infty_c(W\times (0, T))$, by Proposition 2.2 (i)
	we have
	$$f:= -\mathcal{R}^s_A g|_{\Omega\times (0, T)}
	= -(-\Delta)^s g|_{\Omega\times (0, T)},$$
	which is smooth over $\bar{\Omega}\times [0,T]$.
	
	Based on Proposition 3.1 and 3.2, we have the following well-posedness result
	\begin{prop}
		Let $g\in C^\infty_c(W\times (0, T))$
		Then (\ref{Cafracspacetime}) 
		(or (\ref{RLfracspacetime})) has a unique solution $u$ with $w:= u-g$ satisfying
		$$w\in L^2(0, T; \tilde{H}^s(\Omega))\cap C([0,T]; L^2(\Omega)),\qquad
		I^\alpha_{0,t}w\in H^1(0, T; H^{-s}(\Omega))\cap C([0,T]; L^2(\Omega)).$$
	\end{prop}
	
	We denote the solution operator $g\to u_g$ associated with (\ref{Cafracspacetime}) by $P_{A, q}$. We note that we have the same 
	well-posedness result for the dual problem
	\begin{equation}\label{dualRLfracspacetime}
		\left\{
		\begin{aligned}
			D^\alpha_{t,T} u+ \mathcal{R}^s_{A(t)} u + q(t)u&= 0
			\quad \,\,\,\Omega\times (0, T)\\
			u&= g\quad \,\,\,\Omega_e\times (0, T)\\
			I^\alpha_{t,T}u&= 0\quad \,\,\,\Omega\times \{T\}
		\end{aligned}
		\right.
	\end{equation}
	and we denote the associated solution operator by $P^*_{A, q}$.
	
	\subsection{Inverse problem}
	\subsubsection{DN map}
	Recall that we defined the Dirichlet-to-Neumann map $\Lambda_{A, q}$ associated with (\ref{Cafracspacetime}) in (\ref{tDN}). We  also define
	the Dirichlet-to-Neumann map $\Lambda^*_{A, q}$ associated with
	(\ref{dualRLfracspacetime}) by 
	\begin{equation}\label{dualtDN}
		\Lambda^*_{A, q}g:= \mathcal{R}^s_A (P^*_{A, q}g)|_{\Omega_e\times (0, T)}.
	\end{equation}
	The well-posed results ensure that $\Lambda_{A, q}$ and $\Lambda^*_{A, q}$
	are well-defined at least for 
	$g\in C^\infty_c(W\times (0, T))$.
	
	Let $g\in C^\infty_c(W_1\times (0, T))$ and
	$h\in C^\infty_c(W_2\times (0, T))$.
	
	By the definition of $B_t$, $P_{A, q}$ and $\Lambda_{A, q}$ we have
	$$\int^T_{0}\langle \Lambda_{A, q}g(t), h(t)\rangle\,dt
	= \int^T_{0}\langle \mathcal{R}^s_{A(t)}u(t), \tilde{h}(t)\rangle\,dt
	- \int^T_{0}\langle \mathcal{R}^s_{A(t)}u(t), \tilde{h}(t)- h(t)\rangle\,dt$$
	$$=  \int^T_{0}\langle \mathcal{R}^s_{A(t)}u(t), \tilde{h}(t)\rangle\,dt+
	\int^T_{0}\langle \partial^\alpha_t u(t)+ q(t)u(t), \tilde{h}(t)- h(t)\rangle\,dt$$
	\begin{equation}\label{DNint1}
		= \int^T_{0}B_t[u(t), \tilde{h}(t)]\,dt
		+ \int^T_{0}\langle\partial^\alpha_t u(t), \tilde{h}(t)\rangle_\Omega\,dt\
	\end{equation}
	for any $\tilde{h}$ satisfying 
	$\tilde{h}- h\in L^2(0, T; \tilde{H}^s(\Omega))$.
	Here $u:= P_{A, q}g$, $w:= u- g$ and 
	$$\langle\partial_t u(t), \tilde{h}(t)\rangle_\Omega=
	\langle\partial_t w(t), \tilde{h}(t)- h(t)\rangle.$$
	
	Similarly we have
	\begin{equation}\label{DNint2}
		\int^T_{0}\langle \Lambda^*_{A, q}h(t), g(t)\rangle\,dt
		= \int^T_{0}B_t[u^*(t), \tilde{g}(t)]\,dt
		+ \int^T_{0}\langle D^\alpha_{t,T}u^*(t), \tilde{g}(t)\rangle_\Omega\,dt
	\end{equation}
	where $u^*:= P^*_{A, q}h$ for any $\tilde{g}$ satisfying 
	$\tilde{g}- g\in L^2(0, T; \tilde{H}^s(\Omega))$.
	
	\begin{prop}
		For $g\in C^\infty_c(W_1\times (0, T))$ and
		$h\in C^\infty_c(W_2\times (0, T))$, we have
		\begin{equation}\label{dualsym}
			\int^T_{0}\langle \Lambda_{A, q}g(t), h(t)\rangle\,dt
			= \int^T_{0}\langle \Lambda^*_{A, q}h(t), g(t)\rangle\,dt.
		\end{equation}
	\end{prop}
	\begin{proof}
		Let $\tilde{h}= u^*$ in (\ref{DNint1}) and 
		let $\tilde{g}= u$ in (\ref{DNint2}).
		Since $u(0)= I^\alpha_{t,T}u^*(T)= 0$, we have
		$$\int^T_{0}\langle \Lambda_{A, q}g(t), h(t)\rangle\,dt
		- \int^T_{0}\langle \Lambda^*_{A, q}h(t), g(t)\rangle\,dt$$
		$$= \int^T_{0}\langle\partial^\alpha_t u(t), u^*(t)\rangle_\Omega
		- \langle D^\alpha_{t,T}u^*(t), u(t)\rangle_\Omega\,dt 
		=\langle u(t), I^\alpha_{t,T}u^*(t)\rangle_\Omega|^{t= T}_{t= 0}= 0$$
		based on the symmetry of $B_t$ and the integration by parts formula (\ref{fracibp}).
	\end{proof}
	
	Now we are ready to derive the integral identities for Dirichlet-to-Neumann maps.
	
	For $g_j\in C^\infty_c(W_j\times (0, T))$ ($j= 1, 2$), 
	let $u_1= P_{A_1, q_1}(g_1)$
	and $u^*_2= P^*_{A_2, q_2}(g_2)$,\\
	i.e. $u_1$ is the solution of
	\begin{equation}
		\left\{
		\begin{aligned}
			\partial_t^\alpha u+ \mathcal{R}^s_{A_1(t)} u+ q_1(t)u &= 0\quad \,\,\,\text{in}\,\,\Omega\times (0, T)\\
			u&= g_1 \quad \text{in}\,\,\Omega_e\times (0, T)\\
			u&= 0\quad \,\,\,\text{in}\,\,\Omega\times \{0\}\\
		\end{aligned}
		\right.
	\end{equation}
	and $u^*_2$ is the solution of
	\begin{equation}
		\left\{
		\begin{aligned}
			D^\alpha_{t,T} u+ \mathcal{R}^s_{A_2(t)} u+ q_2(t)u &= 0\quad \,\,\,\text{in}\,\,\Omega\times (0, T)\\
			u&= g_2 \quad \text{in}\,\,\Omega_e\times (0, T)\\
			I^\alpha_{t,T}u&= 0\quad \,\,\,\text{in}\,\,\Omega\times \{T\}.\\
		\end{aligned}
		\right.
	\end{equation}
	Then by (\ref{DNint1}), (\ref{DNint2}), (\ref{fracibp}) and Proposition 3.4 we have
	$$\int^T_{0}\langle \Lambda_{A_1, q_1}g_1(t), g_2(t)\rangle-
	\langle\Lambda_{A_2, q_2}g_1(t), g_2(t)\rangle\,dt$$
	$$= \int^T_{0}\langle \Lambda_{A_1, q_1}g_1(t), g_2(t)\rangle\,dt
	-\int^T_{0}\langle \Lambda^*_{A_2, q_2}g_2(t), g_1(t)\rangle\,dt$$
	$$= \int^T_{0}B^{(1)}_t[u_1(t), u^*_2(t)]
	+ \langle\partial^\alpha_t u_1(t), u^*_2(t)\rangle_\Omega\,dt
	- \int^T_{0}B^{(2)}_t[u^*_2(t), u_1(t)]
	+ \langle D^\alpha_{t,T}u^*_2(t), u_1(t)\rangle_\Omega\,dt$$
	$$=\int^T_{0}B^{(1)}_t[u_1(t), u^*_2(t)]\,dt
	- \int^T_{0}B^{(2)}_t[u_1(t), u^*_2(t)]\,dt$$
	\begin{equation}\label{intdif}
		= \int^T_{0}\int_{\mathbb{R}^n}\int_{\mathbb{R}^n}
		G(x, y, t)u_1(y, t)u^*_2(x, t)
		-\int^T_{0}\int_\Omega(q_2- q_1)u_1u^*_2
	\end{equation}
	where
	$$G(x, y, t):= 2(R_{A_2(t)}(x, y)- R_{A_1(t)}(x, y))K(x, y).$$
	
	\subsubsection{Runge approximation}
	\begin{prop}
		Suppose $\Omega\subset B_r(0)$ for some constant $r> 0$ and $W$ is a nonempty open set in $\mathbb{R}^n$ s.t. $W\cap B_{3r}(0)= \emptyset$. Then 
		$$S:= \{P_{A, q}g|_{\Omega\times (0, T)}: 
		g\in C^\infty_c(W\times (0, T))\},$$
		$$S^*:= \{P^*_{A, q}g|_{\Omega\times (0, T)}: 
		g\in C^\infty_c(W\times (0, T))\}$$
		are dense in $L^2(\Omega\times (0, T))$.
	\end{prop}
	\begin{proof}
		By the Hahn-Banach Theorem, it suffices to show that:
		
		If $f\in L^2(\Omega\times (0, T))$ and $\int^T_{0}\int_\Omega wf= 0$ for all $w\in S$, then $f= 0$ in $\Omega\times (0, T)$.
		
		In fact, for an arbitrary given $f\in L^2(\Omega\times (0, T))$, by (the dual version of) Proposition 3.1 we know that the solution of
		\begin{equation}
			\left\{
			\begin{aligned}
				D^\alpha_{t,T} v+ \mathcal{R}^s_{A(t)} v+ q(t)v &= f\quad \text{in}\,\,\Omega\times (0, T)\\
				I^\alpha_{t,T}v&= 0\quad \text{in}\,\,\Omega\times \{T\}.\\
			\end{aligned}
			\right.
		\end{equation}
		satisfies
		$$v\in L^2(0, T; \tilde{H}^s(\Omega)),\qquad
		I^\alpha_{t,T}v\in H^1(0, T; H^{-s}(\Omega))\cap C([0,T]; L^2(\Omega)).$$
		
		For $g\in C^\infty_c(W\times (0, T))$, 
		write $u_g:= P_{A, q}g$, then we have
		$$\int^T_{0}\int_\Omega  u_g f= \int^T_{0}\langle D^\alpha_{t,T}v(t)+ 
		\mathcal{R}^s_{A(t)} v(t)+ q(t)v(t), u_g(t)- g(t)\rangle\,dt$$
		$$= \int_{0}^T\langle\partial^\alpha_t u_g(t), v(t)\rangle
		+ B_t[u_g(t), v(t)]\,dt
		- \int_{0}^T\langle\mathcal{R}^s_{A(t)}g(t), v(t)\rangle\,dt$$
		\begin{equation}\label{RAPid}
			= -\int_{0}^T\langle\mathcal{R}^s_{A(t)}v(t), g(t)\rangle\,dt.
		\end{equation}
		The first equality holds since $u_g- g\in L^2(0, T; \tilde{H}^s(\Omega))$;
		The second equality holds since $u_g(0)= I^\alpha_{t,T}v(T)= 0$ ensures
		$$\int_{0}^T\langle D^\alpha_{t,T}v(t), u_g(t)\rangle=
		\int_{0}^T\langle\partial^\alpha_t u_g(t), v(t)\rangle;$$
		The last equality holds since $v\in L^2(0, T; \tilde{H}^s(\Omega))$ and $u_g$ is the solution of (\ref{Cafracspacetime}).
		
		Hence, if $\int^T_{0}\int_\Omega wf= 0$ for all $w\in S$, then 
		(\ref{RAPid}) yields
		$$\int_{0}^T\langle\mathcal{R}^s_{A(t)}v(t), g(t)\rangle\,dt= 0,\qquad
		g\in C^\infty_c(W\times (0, T))$$
		so for each $t$ we have
		$$v(t)\in \tilde{H}^s(\Omega),
		\qquad \mathcal{R}^s_{A(t)}v(t)|_W= 0,$$
		which implies $v(t)= 0$ in $\mathbb{R}^n$ for each $t$ by Proposition 2.2 (ii) and thus $f= 0$ in $\Omega\times (0, T)$. 
		
		Similarly we can show $S^*$ is dense in $L^2(\Omega\times (0, T))$.
	\end{proof}
	We remark that a similar argument has been used to prove the approximate controllability result for $\partial^\alpha_t+ (-\Delta)^s$ (see Theorem 2.6 in \cite{Warma2019approx}).
	
	\subsubsection{Proof of Theorem 1.1}
	We note that time-fractional derivatives do not appear in the integral identity
	(\ref{intdif}) so we can use the same argument as the one in \cite{li2021fractional} to prove Theorem 1.1 based on the Runge approximation property (Proposition 3.5). To avoid redundancy, we will skip inessential steps in the proof. We refer readers to Subsection 4.3 in \cite{li2021fractional} for full details.
	
	\begin{proof}
		We write  $u_1= P_{A_1, q_1}(g_1)$ and $u^*_2= P^*_{A_2, q_2}(g_2)$ 
		for $g_j\in C^\infty_c(W_j\times (0, T))$.
		
		Based on (\ref{intdif}) and the assumptions stated in Theorem 1.1, we have
		\begin{equation}\label{Gid}
			\int^T_{0}\int_\Omega\int_\Omega G(x, y, t)u_1(y, t)u^*_2(x, t)
			=\int^T_{0}\int_\Omega(q_2- q_1)u_1u^*_2.
		\end{equation}
		
		\textbf{Determine $A$}: We fix open sets $\Omega_j\subset \Omega$ s.t. 
		$\Omega_1\cap \Omega_2= \emptyset$. 
		We fix $\phi_j\in C^\infty_c(\Omega_j)$ and the constants
		$a, b\in (0, T)$. We define
		$$\tilde{\phi}_j(x, t):= 1_{[a, b]}(t)\phi_j(x).$$
		
		By Proposition 3.5, we can choose
		$g_1\in C^\infty_c(W_1\times (0, T))$ (resp. $g_2\in C^\infty_c(W_2\times (0, T))$) s.t. $u_1$ (resp. $u^*_2$) approximates
		$\tilde{\phi}_1$ (resp. $\tilde{\phi}_2$) arbitrarily in $L^2(\Omega\times (0, T))$.
		
		The key observation is that $\phi_1\phi_2= 0$ while $\phi_1\otimes\phi_2\neq 0$ in general. This enables us to take the limit on both sides of (\ref{Gid}) to get
		$$\int^b_a\int_{\Omega_1}\int_{\Omega_2} G(x, y, t)
		\phi_1(y)\phi_2(x)\,dxdydt= 0.$$
		Since the choices of $\phi_1, \phi_2$ and $[a, b]$ are arbitrary, we can conclude that
		$G(x, y, t)= 0$ for $x, y \in\Omega$ whenever $x\neq y$. Thus we know that
		\begin{equation}\label{RA12}
			R_{A_1(t)}(x, y)= R_{A_2(t)}(x, y),\quad x, y\in\Omega
		\end{equation}
		for each $t$. Then we can show that $A_1(t)= \pm A_2(t)$.
		
		\textbf{Determine $q$}: Now (\ref{Gid}) becomes
		$$\int^T_{0}\int_\Omega(q_2- q_1)u_1u^*_2= 0.$$
		For $f\in L^2(\Omega\times (0, T))$, 
		again by Proposition 3.5, we can choose
		$g_1\in C^\infty_c(W_1\times (0, T))$ (resp. $g_2\in C^\infty_c(W_2\times (0, T))$) s.t. $u_1$ (resp. $u^*_2$) approximates
		$f$ (resp. constant function $1$) arbitrarily in $L^2(\Omega\times (0, T))$.
		Then we take the limit to get
		$$\int^T_{0}\int_\Omega(q_1-q_2)f= 0.$$
		We conclude that $q_1= q_2$ since the choice of $f$ is arbitrary.
	\end{proof}
	
	We remark that if we are only interested in the problem  
	\begin{equation}\label{Laspacetime}
		\left\{
		\begin{aligned}
			\partial^\alpha_t u+ (-\Delta)^s u + q(t)u&= 0
			\quad \,\,\,\Omega\times (0, T)\\
			u&= g\quad \,\,\,\Omega_e\times (0, T)\\
			u&= 0\quad \,\,\,\Omega\times \{0\}
		\end{aligned}
		\right.
	\end{equation}
	instead of the more general (\ref{Cafracspacetime}), then the inverse problem will be reduced to the determination of $q$ from partial knowledge of $\Lambda_q$ (i.e. $\Lambda_{0, q}$), and in this case the assumptions on $W_j$ in the statement of Theorem 1.1 are not necessary for the unique determination. 
	
	In fact, based on the unique continuation property of fractional Laplacian (Proposition 2.3) instead of Proposition 2.2, we can use the same argument as the one in the proof of Proposition 3.5 to prove the following Runge approximation property.
	
	\begin{prop}
		Let $W\subset \Omega_e$ be nonempty and open. Then
		$$S:= \{P_q g|_{\Omega\times (0, T)}: 
		g\in C^\infty_c(W\times (0, T))\},\qquad
		S^*:= \{P^*_q g|_{\Omega\times (0, T)}: 
		g\in C^\infty_c(W\times (0, T))\}$$
		are dense in $L^2(\Omega\times (0, T))$. (Here $P_q:= P_{0, q}$ and
		$P^*_q:= P^*_{0, q}$.)
	\end{prop}
	
	Then we can use the argument in the second half of the proof of Theorem 1.1 to prove the following unique determination theorem. 
	
	\begin{prop}
		Let $W_1, W_2\subset \Omega_e$ be nonempty and open. If
		$$\Lambda_{q_1} g|_{W_2\times (0, T)}= \Lambda_{q_2} g|_{W_2\times (0, T)},
		\qquad g\in C^\infty_c(W_1\times (0, T)),$$
		then $q_1= q_2$ in $\Omega\times (0, T)$.
	\end{prop}
	
	\section{Semilinear problem}
	\subsection{Forward problem}
	Now we turn to the semilinear problem (\ref{Semifracspacetime}).
	We first prove the following a priori $L^\infty$ estimate.
	\begin{prop}
		Let $g\in C^\infty_c(\Omega_e\times (0, T))$. 
		Suppose $u\in C([0, T]; L^2(\mathbb{R}^n))$ is a solution of 
		\begin{equation}\label{fgfracpara}
			\left\{
			\begin{aligned}
				\partial^\alpha_t u+ (-\Delta)^s u + a(x, t, u)&= f
				\quad \,\,\,\Omega\times (0, T)\\
				u&= g\quad \,\,\,\Omega_e\times (0, T)\\
				u&= 0\quad \,\,\,\Omega\times \{0\}.
			\end{aligned}
			\right.
		\end{equation}
		Then we have
		$$||u||_{L^\infty}\leq \frac{T^\alpha}{\Gamma(\alpha+1)}||f||_{L^\infty(\Omega\times (0, T))}
		+ ||g||_{L^\infty(\Omega_e\times (0, T))}.$$
	\end{prop}
	
	\begin{proof}
		We fix $\varphi\in C^\infty_c(\mathbb{R}^n)$ s.t. $0\leq \varphi\leq 1$
		and $\varphi= 1$ on $\bar\Omega \cup \mathrm{supp}_x g$. We define
		$$\tilde{\varphi}(x, t):= (||f||_{L^\infty(\Omega\times (0,T))}
		\frac{t^\alpha}{\Gamma(\alpha+1)}
		+ ||g||_{L^\infty(\Omega_e\times (0, T))})\varphi(x).$$
		Clearly $(-\Delta)^s\varphi\geq 0$ in $\Omega$ from the pointwise definition 
		(\ref{fracLa}) of $(-\Delta)^s$. Also note that
		$$\partial^\alpha_t(t^\alpha)= \Gamma(\alpha+1),\qquad
		\partial^\alpha_t(c)= 0$$
		so we have
		$$\partial^\alpha_t\tilde{\varphi}+(-\Delta)^s\tilde{\varphi}+ a(x, t, \tilde{\varphi})\geq ||f||_{L^\infty(\Omega\times (0, T))}$$
		in $\Omega\times (0, T)$. Now we consider $\tilde{u}:= u- \tilde{\varphi}$.
		Note that $\tilde{u}\leq 0$ in 
		$\Omega_e\times (0, T)$, $\tilde{u}\leq 0$ at $t= 0$ and 
		\begin{equation}\label{tildeest}
			\partial^\alpha_t\tilde{u}+(-\Delta)^s\tilde{u}
			+ a(x, t, u)- a(x, t, \tilde{\varphi})\leq 0
		\end{equation}
		in $\Omega\times (0, T)$. We write $\tilde{u}= \tilde{u}^+ - \tilde{u}^-$
		where $\tilde{u}^{\pm}= \max\{\pm \tilde{u}, 0\}$. Then $\tilde{u}^+= 0$ in $\Omega_e\times (0, T)$ and $\tilde{u}^+= 0$ at $t=0$. 
		
		If we choose $H(y)= \frac{1}{2}(y^+)^2$ in (\ref{Hconvex}), then we have
		\begin{equation}\label{tildeest1}
			\frac{1}{2}\partial^\alpha_t(||\tilde{u}^+(t)||^2_{L^2(\Omega)})
			\leq \langle \partial^\alpha_t\tilde{u}(t), \tilde{u}^+(t)\rangle.
		\end{equation}
		Note that by the definition of the Caputo derivative, we have
		$$\frac{1}{2}\partial^\alpha_t(||\tilde{u}^+(t)||^2_{L^2(\Omega)})
		= \frac{1}{2}\partial_t(\phi_{1-\alpha} *(||\tilde{u}^+(t)||^2_{L^2(\Omega)}- ||\tilde{u}^+(0)||^2_{L^2(\Omega)}))
		= \frac{1}{2}\partial_t(\phi_{1-\alpha} *(||\tilde{u}^+( t)||^2_{L^2(\Omega)})$$
		since $\tilde{u}^+= 0$ at $t=0$. Also note that
		$u- g\in C([0, T]; L^2(\Omega))$ implies
		$$\phi_{1-\alpha}* ||\tilde{u}^+(t)||^2_{L^2(\Omega)}= 0$$
		at $t= 0$, so by (\ref{alpha1}) we have
		$$\phi_{\alpha}* \partial_t(\phi_{1-\alpha}* ||\tilde{u}^+(t)||^2_{L^2(\Omega)})= \partial_t(\phi_{\alpha}* \phi_{1-\alpha}
		*||\tilde{u}^+(t)||^2_{L^2(\Omega)})$$
		\begin{equation}\label{tildeid1}
			= \partial_t(1* ||\tilde{u}^+(t)||^2_{L^2(\Omega)})= ||\tilde{u}^+
			(t)||^2_{L^2(\Omega)}.
		\end{equation}

		Now we let both sides of (\ref{tildeest}) act on $\tilde{u}^+$. Based on (\ref{tildeest1}) we can apply the convolution operation $\phi_{\alpha}\,*$ and
		(\ref{tildeid1}) to get
		$$\frac{1}{2}||\tilde{u}^+(t)||^2_{L^2(\Omega)}+ \phi_{\alpha}* \langle (-\Delta)^s\tilde{u}(t), \tilde{u}^+(t)\rangle$$
		$$+ \phi_{\alpha}*\langle a(\cdot, t, u(\cdot, t))- a(\cdot, t, \tilde{\varphi}(\cdot, t)), \tilde{u}^+(\cdot, t)\rangle\leq 0$$
		since $\phi_{\alpha}$ is a positive function.
		Note that $a(x, t, u)- a(x, t, \tilde{\varphi})$ has the same sign as $\tilde{u}$ and 
		\begin{equation}\label{PlusfracLa}
			\langle (-\Delta)^s\tilde{u}(t), \tilde{u}^+(t)\rangle\geq 0
		\end{equation}
		Hence we have $||\tilde{u}^+(t)||^2_{L^2(\Omega)}\leq 0$ and thus the only possibility is $\tilde{u}^+= 0$, i.e. $u\leq \tilde{\varphi}$ in $\Omega\times (0, T)$. 
		
		Also note that $a(x, t, \tilde{\varphi})+ a(x, t, u)$ has the same sign as $\tilde{\varphi}+ u$ so similarly we can consider $\tilde{u}:= -u-\tilde{\varphi}$ and show $-\tilde{\varphi}\leq u$ in $\Omega\times (0, T)$. Hence we have
		$|u|\leq \tilde{\varphi}$ in $\Omega\times (0, T)$.
	\end{proof}
	
	We remark that we have
	$$\langle (-\Delta)^s\tilde{u}(t), \tilde{u}^+(t)\rangle=
	\int_{\mathbb{R}^n}\int_{\mathbb{R}^n}\frac{(\tilde{u}(x,t)- \tilde{u}(y,t))(\tilde{u}^+(x,t)- \tilde{u}^+(y,t))}{|x-y|^{n+2s}}dxdy.$$
	Since $(\tilde{u}^+(x,t)- \tilde{u}^+(y,t))(\tilde{u}^-(x,t)- \tilde{u}^-(y,t))\leq 0$, we have 
	$$\langle (-\Delta)^s\tilde{u}(t), \tilde{u}^+(t)\rangle
	\geq \int_{\mathbb{R}^n}\int_{\mathbb{R}^n}\frac{(\tilde{u}^+(x,t)- \tilde{u}^+(y,t))^2}{|x-y|^{n+2s}}dxdy\geq 0.$$
	Hence (\ref{PlusfracLa}) holds true.
	
	To study the well-posedness of (\ref{Semifracspacetime}),
	we again use the substitutions $u= w+g$ and 
	$f= -(-\Delta)^s g|_{\Omega\times (0, T)}$ so (\ref{Semifracspacetime})
	can be converted into the initial value problem
	\begin{equation}\label{Capsemiivp}
		\left\{
		\begin{aligned}
			\partial^\alpha_t w+ (-\Delta)^s w + a(x, t, w)&= f
			\quad \,\,\,\Omega\times (0, T)\\
			w&= 0\quad \,\,\,\Omega\times \{0\}.
		\end{aligned}
		\right.
	\end{equation}
	Our assumptions on $g$ and $a$ ensure that $f-a$ satisfies the conditions 
	(F1), (F2) (and thus (F3), (F4)) on Page 66 and the condition (F5) on Page 81 in \cite{gal2020fractional}. Hence we can apply Theorem 3.2.2, Corollary 3.2.3 and (3.2.32) in \cite{gal2020fractional}
	to obtain the following well-posedness result.
	\begin{prop}
		There exists $0< T_{max}\leq \infty$ s.t. for $0< T < T_{max}$, (\ref{Semifracspacetime}) has a unique solution
		$$w\in C([0, T]; L^\infty(\Omega)).$$ Furthermore, We have either (i) $T_{max}= \infty$ or (ii) $0< T< \infty$ and $lim_{t\to T_{max}}||u(\cdot, t)||_{L^\infty}= \infty$. 
	\end{prop}
	We remark that the solution $w$ actually belongs to certain H\"older type function spaces. We refer readers to Subsection 3.2 in \cite{gal2020fractional} for more details.
	
	By the previous $L^\infty$
	estimate, we have $T_{max}= \infty$ for $g\in C^\infty_c(\Omega_e\times (0, T))$ under our assumptions on the nonlinearity $a$. Hence we do not need any smallness assumptions on the constant $T$.
	
	Now we prove a linearization result.
	
	For $g\in C^\infty_c(\Omega_e\times (0, T))$, we use 
	$u_g$ to denote the solution of the linear problem 
	\begin{equation}\label{linearfracspacetime}
		\left\{
		\begin{aligned}
			\partial^\alpha_t u+ (-\Delta)^s u + a_1(x, t)u&= 0
			\quad \,\,\,\Omega\times (0, T)\\
			u&= g\quad \,\,\,\Omega_e\times (0, T)\\
			u&= 0\quad \,\,\,\Omega\times \{0\}
		\end{aligned}
		\right.
	\end{equation}
	and we use $u_{\lambda, g}$ to denote the solution of the semilinear problem
	\begin{equation}\label{lambdafrac}
		\left\{
		\begin{aligned}
			\partial^\alpha_t u+ (-\Delta)^s u + a(x, t, u)&= 0\quad \,\,\,\Omega\times (0, T)\\
			u&= \lambda g\,\,\,\,\, \Omega_e\times (0, T)\\
			u&= 0\quad \,\,\,\Omega\times \{0\}
		\end{aligned}
		\right.
	\end{equation}
	in the rest of this section.
	
	\begin{prop}
		Let $w_{\lambda, g}:= u_g- \frac{u_{\lambda, g}}{\lambda}$. Then $\lim_{\lambda\to 0}w_{\lambda, g}= 0$
		in $C([0,T]; L^\infty(\Omega))$.
	\end{prop}
	\begin{proof}
		Note that $w_{\lambda, g}= 0$ in $\Omega_e\times (0, T)$ and
		$$\partial_t w_{\lambda, g}+ (-\Delta)^s w_{\lambda, g}
		+ a_1(x, t)w_{\lambda, g}= \frac{1}{\lambda}
		\sum^m_{k=2}a_k(x, t)|u_{\lambda, g}|^{b_k}u_{\lambda, g}$$
		in $\Omega\times (0, T)$. By the $L^\infty$
		estimate (Proposition 4.1), we have
		$$||w_{\lambda, g}||_{L^\infty}\leq \frac{T^\alpha}{\lambda\Gamma(\alpha+ 1)}
		||\sum^m_{k=2}a_k(x, t)|u_{\lambda, g}|^{b_k}u_{\lambda, g}||_{L^\infty(\Omega\times (0, T))}$$
		and $||u_{\lambda, g}||_{L^\infty}\leq 
		\lambda||g||_{L^\infty(\Omega_e\times (0, T))}$ so
		$$||w_{\lambda, g}||_{L^\infty}
		\leq \frac{T^\alpha}{\Gamma(\alpha+ 1)}\sum^m_{k=2}\lambda^{b_k}||a_k(x,t)||_{L^\infty}||g||^{b_k+1}_{L^\infty(\Omega_e\times (0, T))},$$
		which implies $||w_{\lambda, g}||_{L^\infty}\to 0$ as $\lambda\to 0$.
	\end{proof}
	
	\subsection{Inverse problem}
	We will use the same argument as the one in \cite{li2021fractional} to prove Theorem 1.2 based on the unique continuation property of the fractional Laplacian (Proposition 2.3), the Runge approximation property (Proposition 3.6), the $L^\infty$ estimate (Proposition 4.1) and the linearization result (Proposition 4.3). 
	We remark that this kind of regimes work for solving Calder\'on type inverse problems for a broad class of fractional evolutionary operators involving 
	the fractional Laplacian (see Theorem 2.18 in \cite{kowinverse} for more details).
	To avoid redundancy, we will skip inessential steps in the proof. We refer readers to Subsection 4.2 in \cite{li2021inversediff} for full details.
	
	\begin{proof}
		Based on Proposition 2.3, the assumption (\ref{semiIDN})
		implies that 
		$$u^{(1)}_{\lambda, g}= u^{(2)}_{\lambda, g}=:  u_{\lambda, g}$$
		in $\mathbb{R}^n\times (0, T)$. 
		Hence we have
		\begin{equation}\label{aRid}
			(a^{(1)}_1- a^{(2)}_1)u_{\lambda, g}
			= R^{(2)}_1(x, t, u_{\lambda, g})
			- R^{(1)}_1(x, t, u_{\lambda, g})
		\end{equation}
		in $\Omega\times (0, T)$ where
		$$R^{(i)}_j(x, t, z):= \sum^m_{k=j+1}a^{(i)}_k(x, t)|z|^{b_k}z.$$
		
		Now we note that  
		$$||a^{(1)}_1- a^{(2)}_1||_{L^2(\Omega\times (0, T))}$$
		$$\leq ||a^{(1)}_1- a^{(2)}_1||_{L^\infty}||1- \frac{u_{\lambda, g}}{\lambda}||_{L^2(\Omega\times (0, T))}$$
		\begin{equation}\label{2inftyest}
			+ \frac{1}{\lambda}||(a^{(1)}_1- a^{(2)}_1)
			u_{\lambda, g}||_{L^2(\Omega\times (0, T))}.
		\end{equation}
		For given $\delta> 0$, by Proposition 3.6 we can choose 
		$g\in C^\infty_c(W_1\times (0, T))$ s.t. 
		$$||1- u_g||_{L^2(\Omega)\times (0, T)}\leq \delta$$ 
		and for this chosen $g$, we have 
		\begin{equation}\label{Rungelinear}
			||1- \frac{u_{\lambda, g}}{\lambda}||_{L^2(\Omega\times (0, T))}\leq 2\delta
		\end{equation}
		for small $\lambda$ by Proposition 4.3. Since Proposition 4.1 implies that
		$$||u_{\lambda, g}||_{L^\infty}\leq 
		\lambda||g||_{L^\infty(\Omega_e\times (0, T))},$$
		by (\ref{aRid}) we have
		$$\frac{1}{\lambda}||(a^{(1)}_1(\cdot,t)- a^{(2)}_1(\cdot,t))
		u_{\lambda, g}(\cdot, t)||_{L^2(\Omega)}$$
		$$\leq C'(\sum^m_{k=2}\lambda^{b_k}(||a^{(1)}_k||_{L^\infty}
		+||a^{(2)}_k||_{L^\infty})||g||^{b_k+1}_{L^\infty})\to 0$$
		as $\lambda\to 0$. Then by (\ref{2inftyest}) and (\ref{Rungelinear}) we get
		$$||a^{(1)}_1- a^{(2)}_1||_{L^2(\Omega\times (0, T))}\leq 
		2\delta ||a^{(1)}_1- a^{(2)}_1||_{L^\infty(\Omega\times (0, T))}.$$
		Now we conclude that $a^{(1)}_1= a^{(2)}_1$ since $\delta$ is arbitrary.
		
		Iteratively, once we have shown $a^{(1)}_j= a^{(2)}_j$ ($1\leq j\leq m'-1$), we have 
		$$(a^{(1)}_{m'}(x, t)- a^{(2)}_{m'}(x, t))|u_{\lambda, g}|^{b_{m'}}
		u_{\lambda, g}
		= R^{(2)}_{m'}(x, t, u_{\lambda, g})
		- R^{(1)}_{m'}(x, t, u_{\lambda, g})$$
		in $\Omega\times (0, T)$. 
		Then we can repeat the procedure above to conclude that $a^{(1)}_{m'}= a^{(2)}_{m'}$.
	\end{proof}
	\bibliographystyle{plain}
	{\small\bibliography{Reference8}}

\begin{thebibliography}{10}

\bibitem{banerjee2022calder}
Agnid Banerjee, Venkateswaran~P Krishnan, and Soumen Senapati.
\newblock The {Calder\'on} problem for space-time fractional parabolic
  operators with variable coefficients.
\newblock {\em arXiv preprint arXiv:2205.12509}, 2022.

\bibitem{chen2012space}
Zhen-Qing Chen, Mark~M Meerschaert, and Erkan Nane.
\newblock Space--time fractional diffusion on bounded domains.
\newblock {\em Journal of Mathematical Analysis and Applications},
  393(2):479--488, 2012.

\bibitem{covi2020inverse}
Giovanni Covi.
\newblock An inverse problem for the fractional {Schr\"odinger} equation in a
  magnetic field.
\newblock {\em Inverse Problems}, 36(4):045004, 2020.

\bibitem{covi2022higher}
Giovanni Covi, Keijo M{\"o}nkk{\"o}nen, Jesse Railo, and Gunther Uhlmann.
\newblock The higher order fractional {Calder\'on} problem for linear local
  operators: Uniqueness.
\newblock {\em Advances in Mathematics}, 399:108246, 2022.

\bibitem{feizmohammadi2021fractional}
Ali Feizmohammadi, Tuhin Ghosh, Katya Krupchyk, and Gunther Uhlmann.
\newblock Fractional anisotropic {Calder\'on} problem on closed riemannian
  manifolds.
\newblock {\em arXiv preprint arXiv:2112.03480}, 2021.

\bibitem{ferreira2007determining}
David Dos~Santos Ferreira, Carlos~E Kenig, Johannes Sj\"ostrand, and Gunther
  Uhlmann.
\newblock Determining a magnetic {Schr\"odinger} operator from partial cauchy
  data.
\newblock {\em Communications in mathematical physics}, 271(2):467--488, 2007.

\bibitem{folland1995introduction}
Gerald~B Folland.
\newblock {\em Introduction to partial differential equations}.
\newblock Princeton university press, 1995.

\bibitem{gal2020fractional}
Ciprian~G Gal and Mahamadi Warma.
\newblock {\em Fractional-in-time semilinear parabolic equations and
  applications}.
\newblock Springer, 2020.

\bibitem{ghosh2017calderon}
Tuhin Ghosh, Yi-Hsuan Lin, and Jingni Xiao.
\newblock The {Calder\'on} problem for variable coefficients nonlocal elliptic
  operators.
\newblock {\em Communications in Partial Differential Equations},
  42(12):1923--1961, 2017.

\bibitem{ghosh2020uniqueness}
Tuhin Ghosh, Angkana R{\"u}land, Mikko Salo, and Gunther Uhlmann.
\newblock Uniqueness and reconstruction for the fractional {Calder\'on} problem
  with a single measurement.
\newblock {\em Journal of Functional Analysis}, page 108505, 2020.

\bibitem{ghosh2020calderon}
Tuhin Ghosh, Mikko Salo, and Gunther Uhlmann.
\newblock The {Calder\'on} problem for the fractional {Schr\"odinger} equation.
\newblock {\em Analysis \& PDE}, 13(2):455--475, 2020.

\bibitem{ghosh2021calder}
Tuhin Ghosh and Gunther Uhlmann.
\newblock The {Calder\'on} problem for nonlocal operators.
\newblock {\em arXiv preprint arXiv:2110.09265}, 2021.

\bibitem{helin2020inverse}
Tapio Helin, Matti Lassas, Lauri Ylinen, and Zhidong Zhang.
\newblock Inverse problems for heat equation and space--time fractional
  diffusion equation with one measurement.
\newblock {\em Journal of Differential Equations}, 269(9):7498--7528, 2020.

\bibitem{kian2018global}
Yavar Kian, Lauri Oksanen, Eric Soccorsi, and Masahiro Yamamoto.
\newblock Global uniqueness in an inverse problem for time fractional diffusion
  equations.
\newblock {\em Journal of Differential Equations}, 264(2):1146--1170, 2018.

\bibitem{kojima2021existence}
Mizuki Kojima.
\newblock The existence and regularity theory for abstract semilinear
  time-fractional evolution equations.
\newblock {\em arXiv preprint arXiv:2106.00362}, 2021.

\bibitem{kowinverse}
Pu-Zhao Kow and Jenn-Nan Wang.
\newblock Inverse problems for some fractional equations with general
  non-linearity.
\newblock {\em math.ntu.edu.tw}, 2022.

\bibitem{krupchyk2014uniqueness}
Katya Krupchyk and Gunther Uhlmann.
\newblock Uniqueness in an inverse boundary problem for a magnetic
  {Schr\"odinger} operator with a bounded magnetic potential.
\newblock {\em Communications in Mathematical Physics}, 327(3):993--1009, 2014.

\bibitem{lai2020calderon}
Ru-Yu Lai, Yi-Hsuan Lin, and Angkana R{\"u}land.
\newblock The {Calder\'on} problem for a space-time fractional parabolic
  equation.
\newblock {\em SIAM Journal on Mathematical Analysis}, 52(3):2655--2688, 2020.

\bibitem{lai2021inverse}
Ru-Yu Lai and Ting Zhou.
\newblock An inverse problem for non-linear fractional magnetic schrodinger
  equation.
\newblock {\em arXiv preprint arXiv:2103.08180}, 2021.

\bibitem{li2021determining}
Li~Li.
\newblock Determining the magnetic potential in the fractional magnetic
  {Calder\'on} problem.
\newblock {\em Communications in Partial Differential Equations},
  46(6):1017--1026, 2021.

\bibitem{li2021fractional}
Li~Li.
\newblock A fractional parabolic inverse problem involving a time-dependent
  magnetic potential.
\newblock {\em SIAM Journal on Mathematical Analysis}, 53(1):435--452, 2021.

\bibitem{li2021inversediff}
Li~Li.
\newblock An inverse problem for a fractional diffusion equation with
  fractional power type nonlinearities.
\newblock {\em Inverse Problems and Imaging}, 16(3):613--624, 2022.

\bibitem{li2022elastic}
Li~Li.
\newblock On inverse problems arising in fractional elasticity.
\newblock {\em (to appear) Journal of Spectral Theory}, 2022.

\bibitem{meerschaert2002governing}
Mark~M Meerschaert, David~A Benson, Hans-Peter Scheffler, and Peter
  Becker-Kern.
\newblock Governing equations and solutions of anomalous random walk limits.
\newblock {\em Physical Review E}, 66(6):060102, 2002.

\bibitem{nakamura1995global}
Gen Nakamura, Ziqi Sun, and Gunther Uhlmann.
\newblock Global identifiability for an inverse problem for the {Schr\"odinger}
  equation in a magnetic field.
\newblock {\em Mathematische Annalen}, 303(1):377--388, 1995.

\bibitem{quan2022calder}
Hadrian Quan and Gunther Uhlmann.
\newblock The {Calder\'on} problem for the fractional dirac operator.
\newblock {\em arXiv preprint arXiv:2204.00965}, 2022.

\bibitem{vergara2017stability}
Vicente Vergara and Rico Zacher.
\newblock Stability, instability, and blowup for time fractional and other
  nonlocal in time semilinear subdiffusion equations.
\newblock {\em Journal of Evolution Equations}, 17(1):599--626, 2017.

\bibitem{Warma2019approx}
Mahamadi Warma.
\newblock Approximate controllabilty from the exterior of space-time fractional
  diffusive equations.
\newblock {\em arXiv preprint arXiv:1802.08028}, 2019.

\bibitem{zacher2009weak}
Rico Zacher.
\newblock Weak solutions of abstract evolutionary integro-differential
  equations in hilbert spaces.
\newblock {\em Funkcialaj Ekvacioj}, 52(1):1--18, 2009.

\end{thebibliography}
\end{document}